\documentclass[12pt]{amsart}
\usepackage{amsthm}
\usepackage{amsmath,amssymb,euscript,mathrsfs,bm}
\usepackage{eepic,bbm}
\usepackage{bbold}
\usepackage{dsfont}
\usepackage{ebgaramond}
\usepackage{setspace}
\usepackage{tikz-cd}
\usepackage[colorlinks=true,linkcolor=black,citecolor=black,urlcolor=black]{hyperref}
\usepackage{cleveref}

\newcommand{\F}{\mathds{F}}
\newcommand{\C}{\mathds{C}}

\newcommand{\Q}{\mathds{Q}}

\newcommand{\Z}{\mathds{Z}}
\renewcommand{\P}{\mathds{P}}
\newcommand{\A}{\mathds{A}}
\newcommand{\G}{\mathds{G}}

\renewcommand{\O}{\mathcal{O}}

\newcommand{\cL}{\mathcal{L}}

\DeclareMathOperator{\tchar}{char}

\DeclareMathOperator{\im}{im}

\DeclareMathOperator{\Hom}{Hom}

\DeclareMathOperator{\Pic}{Pic}

\DeclareMathOperator{\Spec}{Spec}

\newcommand{\dual}{\vee}      %

\newcommand{\kk}{\mathds{k}}        %

\newtheoremstyle{scthm}%
{}{}{\itshape}{}{\scshape}{.}{ }{}

\newtheoremstyle{scdef}%
{}{}{}{}{\scshape}{.}{ }{}

\theoremstyle{scthm}

\newtheorem{theorem}{Theorem}
\newtheorem{lemma}[theorem]{Lemma}
\newtheorem{proposition}[theorem]{Proposition}
\newtheorem{corollary}[theorem]{Corollary}
\newtheorem{conjecture}[theorem]{Conjecture}

\newtheorem*{thm*}{Theorem}
\newtheorem*{cor*}{Corollary}
\newtheorem*{prop*}{Proposition}
\newtheorem*{claim*}{Claim}
\newtheorem*{lem*}{Lemma}

\theoremstyle{scdef}
\newtheorem{definition}[theorem]{Definition}
\newtheorem{remark}[theorem]{Remark}

\newtheorem*{def*}{Definition}
\newtheorem*{rmk*}{Remark}

\begin{document}

\newcommand{\isom}{\cong}
\renewcommand{\setminus}{\smallsetminus}
\renewcommand{\phi}{\varphi}
\newcommand{\exterior}{\textstyle\bigwedge}
\renewcommand{\tilde}{\widetilde}
\renewcommand{\hat}{\widehat}
\renewcommand{\bar}{\overline}

\title[Chow and Pic of Richardson varieties]{On Chow and Picard groups of Richardson varieties}
\author{David Anderson}
\address{Department of Mathematics, The Ohio State University, Columbus, OH 43210}
\email{anderson.2804@math.osu.edu}

\author{Sam Payne}
\address{Department of Mathematics, University of Michigan, Ann Arbor, MI 48109}
\email{sdpayne@umich.edu}

\thanks{DA was partially supported by NSF CAREER DMS-1945212, and SP was partially supported by NSF DMS-2542134. Both authors were supported by memberships at the Institute for Advanced Study funded by the Charles Simonyi Endowment.}

\date{August 5, 2026}

\begin{abstract}
We show that if $\Pi \subseteq G/P$ is a projected Richardson variety, then the degree map $\Pic(\Pi) \to A_1(\Pi)^\dual$ is injective.  Moreover, it is an isomorphism in characteristic zero and in all but finitely many positive characteristics. We conjecture that it is an isomorphism in every characteristic. Our proofs are geometric, using a linear paving, topological properties of klt singularities, and reduction to characteristic zero.
\end{abstract}

\maketitle

\renewcommand{\bfseries}{\itshape}

The Picard group of a Schubert variety $Y$ in a generalized flag variety $G/P$ is well understood, in every characteristic: the restriction morphism $\Pic(G/P) \to \Pic(Y)$ is surjective, and its kernel has a simple description \cite[\S XII]{mathieu}.  Likewise, the Chow group of curves $A_1(Y)$ is freely generated by the classes of the one-dimensional Schubert varieties in $Y$, because $Y$ is a union of finitely many $B$-orbits, each isomorphic to an affine space.  Moreover, the degree map $\Pic(Y) \to A_1(Y)^\vee$ is an isomorphism.

The situation is more subtle for Richardson varieties.  These varieties need not admit a cell decomposition by affine spaces, nor an action of a solvable group with a dense orbit.  In fact, for a Richardson variety $X \subseteq G/P$, the restriction morphism $\Pic(G/P) \to \Pic(X)$ need not be surjective: already the three-dimensional quadric $Sp_4/P$ contains a Richardson curve that is a conic, so the restricted line bundles on it have even degree.  The purpose of this note is to show that the relation between the two groups nevertheless survives, for Richardson varieties and, more generally, for projected Richardson varieties $\Pi \subseteq G/P$, the images of Richardson varieties in $G/B$ under the projection: the degree map $\Pic(\Pi) \to A_1(\Pi)^\vee$ is injective in every characteristic, and it is an isomorphism in characteristic zero and in all but finitely many positive characteristics.

Our work builds on, and is complementary to, the systematic investigation of Picard groups of Richardson varieties initiated by Buch \cite{buch}.  He proved that, for a Richardson variety $X$ in a cominuscule $G/P$ that is not of exceptional type $E$, the degree map $\Pic(X) \to A_1(X)^\vee$ is an isomorphism.  Along the way, he showed that such an $X$ admits a desingularization $\widetilde X \to X$ such that $\Pic(\widetilde X) \to A_1(\widetilde X)^\vee$ is injective.  We extend both statements to Richardson and projected Richardson varieties in arbitrary $G/P$.

\medskip

Throughout, all varieties are reduced schemes of finite type over an algebraically closed field.  We write $A_k(X)$ for the Chow (homology) groups and $A^k(X)$ for the operational Chow cohomology groups.  % 
``Dual'' means $\Hom$ into the coefficient ring; for example, $A_1(X)^\dual = \Hom(A_1(X),\Z)$.

\medskip

Recall that for any variety $X$, there is a natural degree map $\Pic(X) \to A_1(X)^\dual$ taking the class of a line bundle $\cL$ on $X$ to the homomorphism $[C] \mapsto \deg(\cL|_C)$. 

A {\it Richardson variety} is an intersection of opposite Schubert varieties in $G/P$:
\[
 X_u^v = X_u \cap X^v.
\]
Here $X_u = \overline{B^-uP/P}$ and $X^v = \overline{BvP/P}$, where $B\subset P\subset G$ is a Borel subgroup and $B^-$ is an opposite Borel.  A {\it projected Richardson variety} $\Pi\subseteq G/P$ is the image of a Richardson variety $X_u^v \subseteq G/B$ under the projection $\pi\colon G/B \to G/P$.  Projected Richardson varieties are defined over any algebraically closed field. When the choice of base field is relevant, we write $\Pi_\kk$ to denote a projected Richardson variety over a specified field $\kk$. For the comparison across characteristics in \S\ref{sec:charp} we use an integral model $\Pi_\Z$, built from Chevalley group schemes.

 \begin{theorem} \label{thm:isom}
For any projected Richardson variety $\Pi$, the degree map $\Pic(\Pi) \to A_1(\Pi)^\dual$ is injective. Moreover, if $\tchar(\kk) = 0$ then $\Pic(\Pi_\kk) \to A_1(\Pi_\kk)^\vee$ is an isomorphism.
 \end{theorem}
 
We prove that  $\Pic(X) \to A_1(X)^\dual$ is injective more generally, for complete normal varieties that admit a resolution by a linearly paved variety, as defined below. For projected Richardson varieties over $\C$, surjectivity follows from being of klt type, via topological arguments.

\medskip

In positive characteristic, we prove that the degree map is an isomorphism away from finitely many primes, and the excluded primes can be controlled as follows. Recall that $A_1(\Pi)$ is generated by the classes of torus-invariant curves \cite[Theorem~1]{fmss}, of which there are only finitely many. Thus, Theorem~\ref{thm:isom} implies that $\Pic(\Pi)$ is a finitely generated free abelian group. Taking $\kk = \overline \Q$ and spreading out, we see that there is a number field $K$, a finite set of rational primes $S$, and a set of generators for $\Pic(\Pi_{\overline \Q})$ defined over the ring of $S$-integers of $K$.

\begin{theorem}\label{t.spreading-out}
Let $S$ be a set of primes such that $\Pic(\Pi_{\overline \Q})$ is generated by line bundles defined over the $S$-integers of a number field. Then the inclusion $\Pic(\Pi_\kk) \to A_1(\Pi_\kk)^\vee$ is an isomorphism whenever $\tchar(\kk) \not \in S$. 
\end{theorem}

\begin{corollary} \label{cor.all-char}
If $\Pic(\Pi_{\overline \Q})$ is generated by line bundles defined over $\Z$, then the degree map $\Pic(\Pi_\kk) \to A_1(\Pi_\kk)^\vee$ is an isomorphism for all $\kk$.
\end{corollary}

\begin{conjecture} \label{conj:integral-generators}
For every projected Richardson variety $\Pi$, the Picard group $\Pic(\Pi_{\overline \Q})$ is generated by line bundles defined over $\Z$.
\end{conjecture}

Conjecture~\ref{conj:integral-generators} implies that the degree map $\Pic(\Pi_\kk) \to A_1(\Pi_\kk)^\vee$ is an isomorphism for every projected Richardson variety, over every algebraically closed field.  The conjecture holds in the special case of Schubert varieties, whose Picard groups are generated by restrictions of line bundles from $G/P$ \cite[\S XII]{mathieu}, and for Richardson varieties in cominuscule $G/P$ not of type $E$, whose Picard groups are generated by pullbacks from a minimal cominuscule flag subvariety \cite{buch}. In both cases the generating line bundles are defined over $\Z$.  

Our work avoids the cominuscule and type hypotheses, and also generalizes from Richardson varieties to projected Richardson varieties, at the cost of potentially excluding finitely many characteristics. Our approach to these questions is also quite different from that of Buch. His resolutions yield precise generators for $A_1$ and make the pairing explicit, whereas our argument is uniform across $G/P$ and trades the specific combinatorics for generalities about the geometry of linear pavings and the topology of klt singularities. It seems likely that future progress will involve a combination of methods.

%%%%%%%%%%%%%%%%%%%%%%%%%%%%%%%%%
\section{Linearly paved varieties}
%%%%%%%%%%%%%%%%%%%%%%%%%%%%%%%%%

A {\it paving} of a variety $X$ is a filtration by closed subsets 
\begin{equation} \label{eq:paving}
X=X_0 \supset X_1 \supset X_2 \supset \cdots \supset X_n = \emptyset.
\end{equation} 
We say that \eqref{eq:paving} is a paving by varieties of a given type if each $X_i \setminus X_{i+1}$ is of that type. For instance, \eqref{eq:paving} is a paving by affine spaces if each $X_i \setminus X_{i+1}$ is isomorphic to an affine space $\A^n$.

Totaro \cite{totaro} and Jannsen \cite{jannsen} studied motivic cohomology of two slightly different classes of varieties that are built inductively out of affine spaces.  Here we consider another variation. Roughly speaking, a {\it linearly paved} variety is one satisfying Totaro's conditions defining {\it linear schemes}, except that where Totaro requires a stratification, we only ask for a paving.

\begin{definition} \label{def.lin-paved}
The class of {\it linearly paved varieties} is characterized inductively as follows:
\begin{itemize}
\item any affine space $\A^n$ is linearly paved;
\item the complement of a linearly paved variety embedded as a closed subvariety of $\A^n$ is linearly paved; and
\item any variety which admits a paving by linearly paved varieties is linearly paved.
\end{itemize}
\end{definition}
In other words, the last condition says that if $X$ admits a paving \eqref{eq:paving} where each $X_i \smallsetminus X_{i+1}$ is a linearly paved variety, then $X$ is linearly paved. This notion sits in between Totaro's linear schemes and Jannsen's {\it linear varieties}. Every linear scheme in the sense of Totaro is linearly paved, and every linearly paved variety is a linear variety in the sense of Jannsen. We note that any variety with a paving by affine spaces is linearly paved. In particular, if a nonsingular projective variety has a torus action with finitely many fixed points, then a generic one-parameter subgroup has the same fixed locus and the associated Bia\l{}ynicki-Birula decomposition gives an affine paving \cite{bb}.

Linearly paved varieties satisfy the {\it Chow-K\"unneth property}: if $X$ is linearly paved and $Y$ is any variety, then the natural homomorphism
\begin{equation}\label{e.chow-kunneth}
 A_*(X) \otimes A_*(Y) \to A_*(X \times Y)
\end{equation}
is an isomorphism.  This is proved by induction over the construction of $X$ in Definition~\ref{def.lin-paved}, exactly as in \cite[\S3]{totaro}.  Following Totaro, the induction carries not only the Chow-K\"unneth isomorphism \eqref{e.chow-kunneth} (his {\it weak} property), but also a {\it strong} property supplementing it with a surjectivity statement for the degree-one higher Chow groups of products.  Affine space has the strong property by homotopy invariance \cite[Lemma~2]{totaro}, it is inherited by the complement of a linearly paved closed subvariety of $\A^n$ \cite[Lemma~3]{totaro}, and the localization long exact sequence then gives the weak property \eqref{e.chow-kunneth} for any variety assembled from such pieces \cite[Lemma~4]{totaro}.  Because each step treats only a single closed subvariety and its open complement, the argument uses only a filtration by closed subsets, and not the closure relations of a stratification, and so applies to linearly paved varieties verbatim. It follows, as in the proof of \cite[Theorem~2]{totaro}, that the natural homomorphism
\begin{equation}\label{e.chow-dual}
  A^1(X) \to A_1(X)^\dual
\end{equation}
is an isomorphism for any complete linearly paved variety $X$.  

\begin{lemma} \label{l.isom}
Let $X$ be a nonsingular, complete, linearly paved variety.  Then the homomorphism $\Pic(X) \to A_1(X)^\dual$ is an isomorphism. 
\end{lemma}

\begin{proof}
If $X$ is linearly paved, then \eqref{e.chow-dual} is an isomorphism. Combine this with the fact that the canonical homomorphism $\Pic(X) \to A^1(X)$ is an isomorphism whenever $X$ is nonsingular \cite[Ex.~2.1.1, Cor.~14.4(a)]{it}.
\end{proof}

\begin{theorem} \label{thm:linearly-paved-inj}
Suppose a complete normal variety $X$ admits a desingularization $f\colon \tilde{X} \to X$ by a linearly paved variety $\tilde{X}$.  Then the canonical homomorphism
\[
  \Pic(X) \to A_1(X)^\dual,
\]
given by $[\cL] \mapsto ( C \mapsto \deg(\cL|_C) )$, is injective.
\end{theorem}

\begin{proof}
Consider the diagram
\[
\begin{tikzcd}
 \Pic(\tilde{X}) \ar[r,"\sim"] \ar[from=d,"f^*"] &  A_1(\tilde{X})^\dual \ar[from=d,"f_*^\vee"] \\
 \Pic(X) \ar[r] & A_1(X)^\dual.
\end{tikzcd}
\]
The pullback $f^*\colon \Pic(X) \to \Pic(\tilde{X})$ is injective, because $\cL \cong f_*f^*\cL$ using $f_*\O_{\tilde{X}} \cong \O_X$ (since $X$ is normal and $f$ is proper and birational).  The map $\Pic(\tilde{X}) \to A_1(\tilde{X})^\dual$ is an isomorphism by Lemma~\ref{l.isom}, and the assertion follows. %
\end{proof}

%%%%%%%%%%%%%%%%%%%%%%%%%%%%%%%%%
\section{Richardson and projected Richardson varieties} \label{sec:Richardson}
%%%%%%%%%%%%%%%%%%%%%%%%%%%%%%%%%

Let $G$ denote a connected reductive algebraic group, with maximal torus and Borel subgroup $T \subset B \subset G$.  Let $P \supset B$ be a standard parabolic subgroup and let $B^-$ be the opposite Borel.  Let $W$ be the associated Weyl group and $w_0 \in W$ its longest element.

Recall that a {\it Richardson variety} is an intersection of opposite Schubert varieties in $G/P$:
\[
 X_u^v = X_u \cap X^v.
\]
Here $X_u = \overline{B^-uP/P}$ and $X^v = \overline{BvP/P}$.  The intersection $X_u^v$ is nonempty if and only if $u\leq v$ in Bruhat order. We follow the indexing convention from \cite{kls}: the lower index is the lower endpoint of the Bruhat interval and the upper index is the upper endpoint.

The {\it Richardson cell} corresponding to the Bruhat interval $[u,v]$ is $$(X_u^v)^\circ := B^- uP/P \cap B v P/P.$$ It is an open dense subvariety of $X_u^v$.  
Each Richardson variety $X_u^v$ is preserved by the action of the maximal torus $T \subset B$ and has finitely many $T$-fixed points, indexed by the Bruhat interval $[u,v]$.  Richardson varieties are rational varieties, normal, and have rational resolutions \cite{brion, billey-coskun, kls}.  

Given a Richardson variety $R = X_u^v$, we have the {\it Brion resolution} $\tilde{R} \to R$, which is constructed from a fiber product of Bott-Samelson varieties, as follows.  Given any sequence of simple roots $\underline\gamma = (\gamma_1,\ldots,\gamma_d)$, define
\begin{equation}\label{e.bs}
  Z(\underline\gamma) = G/B \times_{G/P_{\gamma_1}} G/B \times_{G/P_{\gamma_2}} \cdots \times_{G/P_{\gamma_d}} G/B,
\end{equation}
along with $d+1$ projections $\pi_i\colon Z(\underline{\gamma})\to G/B$, for $0 \leq i \leq d$.  (Here $P_\gamma$ is the minimal parabolic subgroup containing $B$ and having $-\gamma$ as a root.)  Let $\underline\beta = (\beta_1,\ldots,\beta_\ell)$ be a reduced word for $v$, considered as a sequence of simple roots.  Let $\underline\alpha = (\alpha_1,\ldots,\alpha_k)$ be a reduced word for $w_0 u$, and write $\underline\alpha^{-1} = (\alpha_k,\ldots,\alpha_1)$ for the reversed word.  The Brion resolution is
\begin{equation}
  \tilde{R} := \pi_0^{-1}(eB) \cap \pi_{k+\ell}^{-1}(w_0 B) \subset Z(\underline\beta,\underline\alpha^{-1}),
\end{equation}
along with $\phi\colon \tilde{R}\to G/B$, the restriction of $\pi_\ell$.  The image of this map is a Richardson variety $R = X_u^v$, and it is an isomorphism over the open Richardson cell $R^\circ \subset R$.  The variety $\tilde{R}$ is sometimes called a \emph{brick manifold}; see \cite{balan,escobar,kls} for further details.

The variety $\tilde{R}$ is projective, nonsingular, and it has finitely many $T$-fixed points; therefore it admits an affine paving by Bia\l ynicki-Birula cells, and in particular it is linearly paved.

Our results are stated and proved for the more general {\it projected Richardson varieties} $\Pi\subseteq G/P$. By definition, $\Pi$ is the image of some Richardson $X_u^v \subseteq G/B$ under the projection $\pi\colon G/B \to G/P$. We refer the reader to \cite{kls} for further details and references, and record the following basic facts. Every Richardson variety is a projected Richardson variety. For each projected Richardson variety $\Pi$, one can choose $u$ and $v$ so that the projection $\pi\colon X_u^v \to \Pi$ is birational. Composing with the Brion resolution, we obtain a desingularization $\tilde{R} \to \Pi$, where $\tilde R$ is linearly paved. 

Recall, e.g. from \cite{kls}, that a resolution of singularities $f \colon Y \to X$ is called \emph{rational} if $Rf_* \mathcal{O}_Y = \mathcal{O}_X$ and $Rf_* \omega_Y = \omega_X$. 

\begin{proposition}[{\cite[Theorem~1.1]{billey-coskun}, \cite[Theorem~4.8]{kls}}] \label{p.kls}
A projected Richardson variety $\Pi$ is a rational variety and the resolution $\tilde{R} \to \Pi$ is a rational resolution. In particular, every projected Richardson variety is normal and projected Richardson varieties over $\C$ have rational singularities.
\end{proposition}

\begin{corollary}
Let $\Pi$ be a projected Richardson variety. Then $\Pic(\Pi) \to A_1(\Pi)^\dual$ is injective.
\end{corollary}

\begin{proof}
This is an immediate consequence of Theorem~\ref{thm:linearly-paved-inj} and Proposition~\ref{p.kls}.
\end{proof}

Applying Lemma~\ref{l.isom} to the smooth, complete linearly paved variety $\tilde{R}$, we obtain:

\begin{corollary}
Any projected Richardson variety $\Pi$ has a desingularization $\tilde{R} \to \Pi$ such that $\Pic(\tilde{R}) \to A_1(\tilde{R})^\dual$ is an isomorphism. 
\end{corollary}

%Thus, the conclusion of \cite[Theorem~1.4]{buch} holds for all projected Richardson varieties.  
This answers a question raised by Buch \cite{buch}.

%%%%%%%%%%%%%%%%%%%%%%%%%%%%%%%%%
\section{Linear paving of projected Richardson varieties}
%%%%%%%%%%%%%%%%%%%%%%%%%%%%%%%%%

In this section, we prove that every projected Richardson variety $\Pi \subset G/P$ is linearly paved. Our proof for Richardson varieties uses the Deodhar decomposition \cite{deodhar}, which expresses a Richardson variety $R \subset G/B$ as a disjoint union of locally closed cells of the form $\A^a \times \G_m^b$. Note that these cells do not assemble into a stratification, because the closure of a cell need not be a union of cells. See  \cite{dudas} and \cite[\S4.3]{speyer}. 

For projected Richardson varieties, we also use the projected Richardson stratification of Knutson--Lam--Speyer \cite{kls}. For $u\leq v$ in $W$, write
\[
  R^\circ(u,v) :=B^-uB/B\cap BvB/B\subset G/B
\]
for the corresponding open Richardson cell.

\begin{lemma}\label{lem:linear-paving-open-Richardson}
Every open Richardson cell in $G/B$ is linearly paved.
\end{lemma}

\begin{proof}
Fix a reduced word for $v$.  Deodhar's decomposition writes $R^\circ(u,v)$ as a finite disjoint union of pieces isomorphic to $\A^a\times \G_m^b$ \cite[Theorem~2.1]{Dudas09}.  These pieces can be ordered as the successive differences in a filtration of
$R^\circ(u,v)$ by closed subsets \cite[Lemma~2.5]{Dudas09}.  Since the pieces are linearly paved, so is $R^\circ(u,v)$.
\end{proof}

Now let $P\supset B$ be a parabolic subgroup, and let $\pi\colon G/B\to G/P$ be the projection.  Let $W_P$ be the Weyl group of $P$,
and let $W^P$ be the set of minimal representatives for $W/W_P$.  For a $P$-Bruhat interval $u\leq_P v$, put
\[
  (\Pi^v_u)^\circ:=\pi\bigl(R^\circ(u,v)\bigr) \quad\text{and}\quad \Pi^v_u:=\bar{(\Pi^v_u)^\circ}.
\]
After replacing a $P$-Bruhat interval by its equivalent representative with upper endpoint in $W^P$, every projected Richardson variety is of this form.

\begin{proposition}\label{prop:linearly-paving-projected-Richardson}
Every projected Richardson variety in $G/P$ is linearly paved.
\end{proposition}

\begin{proof}
Let $\Pi=\Pi^v_u$, with $v\in W^P$.  Knutson--Lam--Speyer show that the open projected Richardson varieties form a finite stratification of $G/P$, and 
\[
  \Pi^v_u=
  \coprod_{u\leq u'\leq_P v'\leq v} (\Pi^{v'}_{u'})^\circ .
\]
They also show that the closure of each stratum is a union of strata \cite[Proposition~3.6]{kls}.

Each stratum in this decomposition is linearly paved.  Indeed, if $u'\leq_P v'$, then
\[
  R^\circ(u',v')\xrightarrow{\sim} (\Pi^{v'}_{u'})^\circ,
\]
by \cite[Lemma~3.1 and Corollary~3.2]{kls}, so the claim follows from Lemma~\ref{lem:linear-paving-open-Richardson}.

Let the strata of $\Pi$ be $S_0,S_1,\ldots,S_N$, ordered so that $S_j\subset \bar S_i\setminus S_i$ implies $i<j$.  Put
\[
  Y_i=\bigcup_{j\geq i} S_j .
\]
By the closure relation above, each $Y_i$ is closed in $\Pi$.  Thus
\[
  \Pi=Y_0\supset Y_1\supset\cdots\supset Y_{N+1}=\emptyset
\]
is a paving by linearly paved varieties.  Hence $\Pi$ is linearly paved.
\end{proof}

%%%%%%%%%%%%%%%%%%%%%%%%%%%%%%%%%
\section{Proof of the main theorem in characteristic zero}
%%%%%%%%%%%%%%%%%%%%%%%%%%%%%%%%%

In this section, we show that the inclusion $\Pic(\Pi_\kk) \hookrightarrow A_1(\Pi_\kk)^\vee$ is an isomorphism when $\tchar(\kk) = 0$, which completes the proof of Theorem~\ref{thm:isom}. Because this property is invariant under extensions of algebraically closed fields, we may assume that the ground field is $\C$ and use techniques from complex geometry and topology.

\begin{proposition} \label{prop:chern-iso}
Let $X$ be a projective rationally connected variety with rational singularities over $\C$.  Then the Chern class homomorphism $c_1\colon \Pic(X) \to H^2(X,\Z)$ is an isomorphism.
\end{proposition}

\begin{proof}
Let $f\colon \tilde{X} \to X$ be a projective desingularization. Note that $\tilde{X}$ is rationally connected.  We then have $H^i(\tilde{X},\O_{\tilde{X}}) = 0$ for all $i>0$, by \cite{kmm}.  Since $X$ has rational singularities, we have $R^if_*\O_{\tilde{X}} = 0$ for $i>0$, and hence $H^i(X,\O_X)=0$ for $i>0$, by the Leray spectral sequence.  Now, pass to the analytic topology and apply the exponential sequence to obtain
\[
  0 = H^1(X,\O_X) \to H^1(X,\O_X^{\times}) \xrightarrow{c_1} H^2(X,\Z) \to H^2(X,\O_X) = 0,
\]
and note that $H^1(X,\O_X^{\times}) \isom \Pic(X)$.
\end{proof}

\begin{corollary}\label{cor:projected-richardson-pic-h2}
Let $\Pi_\C$ be a projected Richardson variety over the complex numbers. Then the Chern class homomorphism $c_1\colon \Pic(\Pi_\C)\longrightarrow H^2(\Pi_\C,\Z)$ is an isomorphism.
\end{corollary}

\begin{proof}
This is an immediate consequence of Propositions~\ref{p.kls} and \ref{prop:chern-iso}.
\end{proof}

\begin{proposition} \label{prop:logFano}
Let $\Pi_\C$ be a projected Richardson variety over $\C$. There is an effective $\Q$-divisor $\Delta$ such that the pair $(\Pi_\C,\Delta)$ is klt. Moreover, $\Delta$ may be chosen so that $-(K_{\Pi_\C} + \Delta)$ is ample.
\end{proposition}

\begin{proof}
Choose a Richardson variety $R$ that projects birationally onto $\Pi_\C$. Recall that $\Pi_\C$ is projective and normal, and $R \to \Pi_\C$ has connected fibers \cite{kls}. By \cite[Theorem~3.1]{kumar-schwede}, there is an effective $\Q$-divisor $\Delta_R$ on $R$ such that $(R,\Delta_R)$ is klt and $-(K_R+\Delta_R)$ is ample. Applying \cite[Corollary~3.3]{fujino-gongyo} to $f\colon R\to \Pi_\C$ gives an effective $\Q$-divisor $\Delta$ on $\Pi_\C$ such that $(\Pi_\C,\Delta)$ is klt and
$
        -(K_{\Pi_\C}+\Delta)
$ 
is ample, as required.
\end{proof}

In what follows, we only use that $(\Pi_\C, \Delta)$ is klt, and not the ampleness of $-(K_{\Pi_\C} + \Delta)$.

\begin{corollary} \label{cor:simply-connected-fibers}
Let $\Pi_\C$ be a projected Richardson variety over the complex numbers, and let $f \colon \widetilde \Pi \to \Pi_\C$ be a projective resolution of singularities. Then all fibers of $f$ are simply connected.
\end{corollary}

\begin{proof}
All fundamental groups are taken in the analytic topology. By Proposition~\ref{prop:logFano}, there is an effective $\Q$-divisor $\Delta$ such that $(\Pi_\C,\Delta)$ is klt.

Fix $p\in\Pi_\C$, and let $E=f^{-1}(p)$. There is an analytic neighborhood $N$ of $E$ in $\widetilde\Pi$ which retracts onto $E$ \cite[\S2.1]{takayama2}. Since $f$ is proper, there is an analytic neighborhood $V$ of $p$ such that
\[
        f^{-1}(V)\subset N.
\]
Applying Takayama's local simple connectedness theorem \cite{takayama2} to the restricted resolution over $V$, we obtain a contractible neighborhood $U\subset V$ of $p$ such that $f^{-1}(U)$ is simply connected.

The retraction $N \to E$ restricts to a retraction $f^{-1}(U)\longrightarrow E$. It follows that $E$ is connected, and the inclusion $E\hookrightarrow f^{-1}(U)$ induces an injection
\[
        \pi_1(E)\hookrightarrow\pi_1(f^{-1}(U))=1.
\]
Hence $E$ is simply connected, as required.
\end{proof}

\begin{proposition} \label{prop:primitive-H2}
Let $f \colon Y\to X$ be a proper surjective morphism of complex algebraic varieties. Assume that every fiber $Y_x$ is simply connected.  
Then
\[
        f^*\colon H^2(X,\Z)\longrightarrow H^2(Y,\Z)
\]
is injective with torsion-free cokernel. 
\end{proposition}

\begin{proof}
Consider the Leray spectral sequence for the constant sheaf $\Z_Y$:
\[
E_2^{p,q}=H^p\bigl(X,R^q f_*\Z_Y\bigr) \Longrightarrow H^{p+q}(Y,\Z).
\]
Since $f$ is proper and surjective with connected fibers, we have $R^0f_*\Z_Y\cong \Z_X$. Also, every fiber is simply connected and hence $H_1(Y_x,\Z)=0$. By the universal coefficient theorem, 
\[
H^1(Y_x,\Z)=0 \quad \mbox{ and } \quad H^2(Y_x,\Z) \cong H_2(Y_x,\Z)^\vee.
\]
In particular, $R^1f_*\Z_Y=0$, and the stalks of $R^2f_*\Z_Y$ are torsion-free.

We now examine total degree $2$. Since $R^1f_*\Z_Y=0$, we have 
$E_2^{p,1}=0$  for all $p$.  There is no non-zero differential entering or leaving from $E_r^{2,0}$ for $r \geq 2$.  Hence
\[
E_\infty^{2,0}=E_2^{2,0}=H^2(X,\Z).
\]
The Leray filtration on $H^2(Y,\Z)$ therefore gives an exact sequence
\[
0\longrightarrow H^2(X,\Z)
\xrightarrow{\,f^*\,}
H^2(Y,\Z)
\longrightarrow E_\infty^{0,2}
\longrightarrow 0.
\]
We claim that $E_\infty^{0,2}$ is torsion-free. To see this, note that there is no incoming differential at $E_r^{0,2}$. The $d_2$-differential leaving $E_2^{0,2}$ has target $E_2^{2,1}=0$, and the only possible later differential is $d_3 \colon E_3^{0,2} \to E_3^{3,0}$. Thus $E_\infty^{0,2}$ is a subgroup of
\[
E_2^{0,2}=H^0(X,R^2f_*\Z_Y).
\]
Since $R^2f_*\Z_Y$ has torsion-free stalks, it follows that $E_\infty^{0,2}$ is torsion-free, as claimed. Hence $$f^* \colon H^2(X,\Z)\to H^2(Y,\Z)$$ is injective and has torsion-free cokernel. 
\end{proof}

\begin{proof}[Proof of Theorem~\ref{thm:isom}]
Choose, as in \cite[\S\S2--3]{kls}, a Richardson variety $R\subset G/B$ whose projection to $\Pi_\C$ is birational, and let $\widetilde R\to R$ be the Brion resolution.  The induced map $f\colon \widetilde R\to \Pi_\C$ is a resolution by a linearly paved variety, so $\Pic(\Pi_\C) \to A_1(\Pi_\C)^\vee$ is injective, by Theorem~\ref{thm:linearly-paved-inj}. 

It remains to prove that $\Pic(\Pi_\C) \to A_1(\Pi_\C)^\vee$ is surjective.  Since $\Pi_\C$ has rational singularities, we have an isomorphism
\[
        \Pic(\Pi_\C)\otimes\Q \longrightarrow A^1(\Pi_\C)\otimes\Q; 
\]
see \cite[Prop.~12.1.4]{kollar-mori} and \cite[p.~23]{totaro}.  Since $\Pi_\C$ is complete and linearly paved, by Proposition~\ref{prop:linearly-paving-projected-Richardson}, the map \eqref{e.chow-dual} identifies $A^1(\Pi_\C)$ with $A_1(\Pi_\C)^\vee$.  Hence $\Pic(\Pi_\C)\otimes\Q \to A_1(\Pi_\C)^\vee\otimes\Q$ is an isomorphism.  It then follows from the injectivity already proved that $\Pic(\Pi_\C)\to A_1(\Pi_\C)^\vee$ has torsion cokernel.

Let $\lambda\in A_1(\Pi)^\vee$.  Choose $m>0$ and $\cL\in\Pic(\Pi)$ whose degree functional is $m\lambda$.  By Lemma~\ref{l.isom}, there is a line bundle $ \tilde{\cL}$ on $\widetilde R$ whose degree functional is $f^*\lambda = \lambda \circ f_*$.  Then $\tilde{\cL}^{\otimes m}$ and
$f^*\cL$ have the same degree on every curve in $\widetilde R$; hence, again by Lemma~\ref{l.isom},
\[
        \tilde{\cL}^{\otimes m} \cong f^*\cL.
\]
Thus $m c_1(\tilde{\cL})=f^*c_1(\cL)$ in $H^2(\widetilde R,\Z)$.

The fibers of $f$ are simply connected by Corollary~\ref{cor:simply-connected-fibers}.  Using Proposition~\ref{prop:primitive-H2}, $f^*$ identifies $H^2(\Pi,\Z)$ with a primitive sublattice of $H^2(\widetilde R,\Z)$. Since $m c_1(\tilde{\cL})$ lies in this sublattice, so does $c_1(\tilde{\cL})$.  Hence $c_1(\tilde{\cL})=f^*\alpha$ for some
$\alpha\in H^2(\Pi,\Z)$.

Using Corollary~\ref{cor:projected-richardson-pic-h2}, choose $\cL'\in \Pic(\Pi)$ with $c_1(\cL')=\alpha$.  Then $\tilde{\cL}$ and
$f^*\cL'$ have the same degree on every curve in $\widetilde R$, so $\tilde{\cL} \cong f^*\cL'$ by Lemma~\ref{l.isom}. We claim that $\lambda$ is the degree functional of $\cL'$.  Let $C\subset \Pi$ be an irreducible curve, and choose an irreducible curve $D \subset \widetilde R$ mapping generically finitely onto $C$, of degree $d$. Then $f_*[D]=d[C]$, and
\[
        d\lambda([C])=(f^*\lambda)([D])=\deg(\tilde{\cL}|_D)
        =\deg(f^*\cL'|_D)=d\deg(\cL'|_C).
\]
Thus $\lambda([C])=\deg(\cL'|_C)$.  Since $A_1(\Pi)$ is generated by curve classes, $\lambda$ lies in the image of $\Pic(\Pi)$, and the theorem follows.
\end{proof}

%%%%%%%%%%%%%%%%%%%%%%%%%%%%%%%%%
\section{Spreading out to positive characteristic} \label{sec:charp}
%%%%%%%%%%%%%%%%%%%%%%%%%%%%%%%%%

We have shown that the degree map $\Pic(\Pi_\kk) \to A_1(\Pi_\kk)^\vee$ is always injective and is an isomorphism when $\tchar(\kk) = 0$.  Here, we consider the conditions under which this inclusion remains an isomorphism over fields of characteristic $p > 0$.

\medskip

We refer to \cite[Exp.~XXII--XXIII]{sga3} or \cite[\S II.1, \S II.13]{jantzen} for details about reductive group schemes over a general base ring.  Fix a root datum, so that the associated split Chevalley group scheme $G_\Z$ is simply connected with maximal torus $T_\Z$; let $B_\Z\subset P_\Z$ be standard Borel and parabolic subgroup schemes.  The associated flag scheme $(G/P)_\Z$ is smooth and projective over $\Z$. The Schubert and opposite Schubert subschemes are the reduced closures of the respective $B_\Z$-orbits and $B^-_\Z$-orbits through the relevant $T_\Z$-fixed points, and the Richardson subschemes $R_\Z$ are their scheme-theoretic intersections. The formation of this intersection commutes with base change.  Since the scheme-theoretic intersection of a Schubert variety with an opposite Schubert variety is reduced, in any characteristic \cite[Lemma~1]{brion-lakshmibai}, the geometric fiber of $R_\Z$ over any algebraically closed field $\kk$ is the Richardson variety $R_\kk$.

Let $\rho$ be the sum of fundamental weights, defining a very ample line bundle on $(G/B)_\Z$, let $\mathcal{L}_\Z$ be its restriction to the Richardson subscheme $R_\Z$, and let $\mathcal{L}_\kk$ be the corresponding line bundle on $R_\kk$ for any field $\kk$.  For any $m\geq0$, the Euler characteristic $\chi(R_{\kk}, \mathcal{L}_{\kk}^{\otimes m})$ is independent of the field $\kk$.  This can be seen via standard monomial theory, or computed directly using localization in equivariant K-theory; see, e.g., \cite{brion-lakshmibai}.  So the Hilbert polynomial of $R_\kk$ with respect to $\mathcal{L}_\kk$ is independent of $\kk$, and it follows that $R_\Z$ is flat over $\Z$.

Now consider the projection $\pi\colon (G/B)_\Z \to (G/P)_\Z$, and let $\Pi_\Z \subseteq (G/P)_\Z$ be the scheme-theoretic image of $R_\Z$.  Given $\Pi$, one can assume $R \to \Pi$ is birational, by choosing a smaller $R$ if necessary.  Each geometric fiber $\Pi_\kk$ is normal and the projection is cohomologically trivial, i.e., $R\pi_*\O_{R_\kk}=\O_{\Pi_\kk}$ \cite[Theorem~4.5, Corollary~4.7]{kls}. In particular $R^{i}(\pi_\kk)_*\O_{R_\kk}=0$ for $i>0$ and every $\kk$. Since $R_\Z$ is flat over $\Z$, the theorem on cohomology and base change implies that the formation of $\O_{\Pi_\Z}=\pi_*\O_{R_\Z}$ commutes with base change. It follows that $\Pi_\Z$ is flat and projective over $\Z$, and its geometric fiber at $\kk$ is the projected Richardson variety $\Pi_\kk$.

The torus-invariant curves in $R_\kk$ are precisely those $T$-invariant curves in $(G/B)_\kk$ joining pairs of $T$-fixed points which lie in $R_\kk$; as such, they are indexed by root-theoretic data that is independent of $\kk$.  The $T$-invariant curves in $\Pi_\kk \subseteq (G/P)_\kk$ are the images of the curves in $R_\kk$, so they are also independent of $\kk$.  Let $C_1,\dots,C_r\subset\Pi_\Z$ be the corresponding curves over $\Z$. Each is isomorphic to $\P^1_\Z$, hence flat over $\Z$ with integral geometric fibers, and for every algebraically closed field $\kk$ the fibers $(C_1)_\kk,\dots,(C_r)_\kk$ are precisely the torus-invariant curves of $\Pi_\kk$. By \cite[Theorem~1]{fmss} their classes $[(C_1)_\kk], \ldots, [(C_r)_\kk]$ generate $A_1(\Pi_\kk)$.

\medskip

For each algebraically closed field $\kk$, consider the surjection 
\[
q_\kk \colon \Z^r \to A_1(\Pi_\kk) \quad \mbox{ given by } \quad e_i \mapsto [(C_i)_\kk].
\]
Set $\Lambda_\kk = \ker q_\kk$, and write $\Lambda_\kk^\perp \subset (\Z^r)^\vee$ for its orthogonal complement. By Theorem~\ref{thm:isom}, the degree map gives an inclusion 
\[
\iota_\kk \colon \Pic(\Pi_\kk) \hookrightarrow \Lambda_\kk^\perp,
\]
 and $\iota_\kk$ is an isomorphism when $\tchar(\kk)=0$.

\begin{lemma}\label{lem:specialization}
For every prime $p$ one has $\Lambda_{\bar\Q}\subseteq\Lambda_{\bar\F_p}$.
\end{lemma}

\begin{proof}
Consider the specialization homomorphism on Chow groups, $\sigma\colon A_1(\Pi_{\bar\Q})\to A_1(\Pi_{\bar \F_p})$ \cite[\S20.3]{it}. Each curve $C_i$ is flat over $\Z$ with integral fibers, so $\sigma\bigl([(C_i)_{\bar\Q}]\bigr)=[(C_i)_{\bar \F_p}]$. Hence $q_{\bar\F_p} = \sigma \circ q_{\bar \Q}$, and  the lemma follows.
\end{proof}

\begin{proof}[Proof of Theorem~\ref{t.spreading-out}]
The conclusion is preserved under extension of algebraically closed fields and holds when $\tchar(\kk)=0$, so we may assume $\kk=\bar\F_p$ with $p\notin S$. Choose line bundles $\cL_1,\dots,\cL_m$ on $\Pi$, defined over the ring $\O_{K,S}$ of $S$-integers of a number field $K$, whose restrictions to $\Pi_{\bar\Q}$ generate $\Pic(\Pi_{\bar\Q})$. For each $i$ and $j$, the degree of $\cL_i|_{(C_j)_\kk}$ is independent of the geometric point $\Spec\kk\to\Spec\O_{K,S}$ \cite[\S20.3]{it}. Hence, Lemma~\ref{lem:specialization} gives
\[
\Lambda_{\bar\Q}^\perp\;\subseteq\;\im \iota_{\bar\F_p}\;\subseteq\;\Lambda_{\bar\F_p}^\perp\;\subseteq\;\Lambda_{\bar\Q}^\perp,
\]
so all three inclusions are equalities. It follows that the degree map $$\iota_{\bar \F_p} \colon \Pic(\Pi_{\bar\F_p})\hookrightarrow \Lambda_{\bar\F_p}^\perp = A_1(\Pi_{\bar \F_p})^\vee$$ is surjective, and therefore an isomorphism, as required.
\end{proof}

\begin{remark}
For each prime $p$, Lemma~\ref{lem:specialization} gives $\Lambda_{\bar\F_p}^\perp \subseteq \Lambda_{\bar\Q}^\perp$.  We therefore have inclusions
\[
\Pic(\Pi_{\bar\F_p}) \;\subseteq\; \Lambda_{\bar\F_p}^\perp \;\subseteq\; \Lambda_{\bar\Q}^\perp \;=\; \Pic(\Pi_{\bar\Q})
\]
inside $(\Z^r)^\vee$.  In particular, for each prime $p$, the Picard group of $\Pi_{\bar\F_p}$ is identified with a subgroup of $\Pic(\Pi_{\bar\Q})$.  Conjecture~\ref{conj:integral-generators} predicts that all of these inclusions are isomorphisms.
\end{remark}

{\it Acknowledgements.}  We are grateful to Anders Buch, whose work computing Picard groups and Chow groups of 1-cycles for Richardson varieties in cominuscule $G/P$ inspired this project. We thank him, in particular, for generously sharing his insights and encouraging us to work out the complementary results presented here.  We also thank J\'anos Koll\'ar and Joaqu\'in Moraga for helpful conversations about the topology of klt singularities.

{\it AI disclosure.}  Generative AI was not used by the authors in this work.

\end{document}